\documentclass[12pt]{amsart}
\usepackage{amssymb,amsmath}
\usepackage{enumerate}
\usepackage{xcolor}
\usepackage{mathrsfs}
\newtheorem{theorem}{Theorem}[section]

\newtheorem{lemma}[theorem]{Lemma}
\newtheorem{corollary}[theorem]{Corollary}
\theoremstyle{definition}
\newtheorem{definition}[theorem]{Definition}

\newcommand{\clb}{\mathscr{B}}

\newcommand{\clm}{\mathcal{M}}

\newcommand{\raro}{\rightarrow}

\newcommand{\vp}{\varphi}

\newcommand{\bd}{\mathbb{D}}
\newcommand{\bt}{\mathbb{T}}
\newcommand{\bc}{\mathbb{C}}

\title[A Bishop-Phelps-Bollob\'{a}s theorem for $H^\infty$]{A Bishop-Phelps-Bollob\'{a}s theorem for bounded analytic functions}

\author[Bala]{Neeru Bala}
\address{Department of Mathematics, Indian Institute of Technology Bombay, Powai, Mumbai, 400076, India}
\email{neerusingh41@gmail.com }

\author[Dhara]{Kousik Dhara}
\address{Weizmann Institute of Science, Rehovot 7610001, Israel}
\email{kousik.dhara@weizmann.ac.il}

\author[Sarkar]{Jaydeb Sarkar}
\address{Indian Statistical Institute, Statistics and Mathematics Unit, 8th Mile, Mysore Road, Bangalore, 560 059, India}
\email{jaydeb@gmail.com, jay@isibang.ac.in}

\author[Sensarma]{Aryaman Sensarma}
\address{Indian Statistical Institute, Statistics and Mathematics Unit, 8th Mile, Mysore Road, Bangalore, 560 059, India}
\email{aryamansensarma@gmail.com}


\numberwithin{equation}{section}

\begin{document}

\begin{abstract}
Let $H^\infty$ denote the Banach algebra of all bounded analytic functions on the open unit disc and denote by $\mathscr{B}(H^\infty)$ the Banach space of all bounded linear operators from $H^\infty$ into itself. We prove that the Bishop-Phelps-Bollob\'{a}s property holds for $\mathscr{B}(H^\infty)$. As an application to our approach, we prove that the Bishop-Phelps-Bollob\'{a}s property also holds for operator ideals of $\mathscr{B}(H^\infty)$.
\end{abstract}
\maketitle

\tableofcontents

\section{Introduction}

The main objective of this paper is to connect two classical concepts of independent interest: the Bishop-Phelps-Bollob\'{a}s theorem and $H^\infty$, the algebra of all bounded analytic functions on the open unit disc $\mathbb{D}$ in the complex plane $\mathbb{C}$. More specifically, we prove that the Banach space of bounded linear operators from $H^\infty$ into itself satisfies the Bishop-Phelps-Bollob\'{a}s property (see Theorem \ref{thm: Main intro}). Here, our formulation of the Bishop-Phelps-Bollob\'{a}s property follows the construction of Acosta, Aron, Garc\'{i}a, and Maestre \cite{acosta} (see Definition \ref{def: BPB space}).

Needless to say that $H^\infty$ plays a distinguished role in the general theory of Banach spaces (cf. \cite{Burgain, Pisier}). Besides, the theory of bounded analytic functions in the context of the Bishop and Phelps theorem \cite{Bishop} also appears to fit appropriately as Lomonosov \cite{Lomonosov} discovered that there exists a complex Banach space $X$ and a closed bounded convex subset $S$ of $X$ with the property that the set of support points of $S$ is empty. In this case, $X$ is the predual of $H^\infty_0$, where $H^\infty_0$ is the Banach space of all functions in $H^\infty$ vanishing at the origin. Lomonosov's counterexample is related to the answer of Bishop and Phelps to the question of Klee \cite{Klee}. From this point of view, our result perhaps complements the existing theories of Bishop-Phelps-Bollob\'{a}s property, notably in the setting of Banach spaces of bounded analytic functions (also see the remark on Asplund spaces and follow-up questions at the end of this paper).

As usual, we treat $H^\infty$ as the commutative unital Banach algebra equipped with the supremum norm $\|\cdot\|_{\infty}$, that is,
\[
H^\infty = \{f \in \text{Hol}(\mathbb{D}): \|f\|_{\infty}:= \sup_{z \in \mathbb{D}} |f(z)| < \infty\}.
\]
This space (also commonly denoted by $H^\infty(\mathbb{D})$) is one of the most important commutative Banach algebras encountered in functional analysis and harmonic analysis. Indeed, the theory of bounded analytic functions is an active research area by itself, and, even more, this lays some foundations for several research areas. Where, on the other hand, the Bishop-Phelps-Bollob\'{a}s theorem (see Theorem \ref{BPB} below) is one of the most powerful and useful results in the theory of Banach spaces (especially in the approximation theory for Banach spaces). The genesis of Bishop-Phelps-Bollob\'{a}s theorem in fact goes back to the seminal work of Bishop and Phelps \cite{Bishop} in 1961. As it is evident in the statement below, the power of the Bishop and Phelps theorem lies in its simplicity.

\begin{theorem}[Bishop and Phelps]\label{thm: BP}
The set of norm attaining functionals on a Banach space $X$ is norm dense in its dual space $X^*$.
\end{theorem}

Throughout this paper, all Banach spaces considered will be over the complex field (as our primary interest is in the Banach space $H^\infty$ over $\mathbb{C}$). We recall that a bounded linear operator $T : X \raro Y$ between Banach spaces $X$ and $Y$ (in short, $T \in \clb(X, Y)$ and $T \in \clb(X)$ if $Y = X$) is \textit{norm attaining} if there exists a unit vector $x_0 \in S_X$ such that
\[
\|T\|_{\clb(X, Y)} =	\|T x_0\|_Y,
\]
where $S_X$ denotes the standard unit sphere of $X$.

In 1970, B\'{e}la Bollob\'{a}s \cite{Bollobas} proved a sharper version (or a “quantitative version”, cf. \cite[page 6086]{Aron1}) of Bishop and Phelps theorem by offering a simultaneous approximation of almost norm-attaining functionals by norm-attaining functionals and almost norm-attainment points by norm-attainment points. Bollob\'{a}s' result is popularly known as the  \textit{Bishop-Phelps-Bollob\'{a}s theorem} \cite[Corollary 3.3]{Cascales}.

\begin{theorem}[Bishop-Phelps-Bollob\'{a}s]\label{BPB}
Let $X$ be a Banach space, $\epsilon \in (0, 1)$, and let $f\in S_{X^*}$. Suppose
\[
|f(x)| \geq 1- \epsilon,
\]
for some $x\in S_X$. Then there exist $g \in S_{X^*}$ and $y \in S_X$ such that
\begin{enumerate}
\item $|g(y)|=1$,
\item $\|x-y\|\leq \sqrt{2\epsilon}$, and
\item $\|f-g\|\leq \sqrt{2\epsilon}$.
\end{enumerate}
\end{theorem}

Note in particular that the set of norm attaining linear functionals on $X$ is dense in $X^*$. Therefore, the Bishop-Phelps-Bollob\'{a}s theorem recovers the classical Bishop and Phelps theorem (see Theorem \ref{thm: BP}). We refer the reader to Aron and Lomonosov \cite{AL} for a rapid but excellent survey on Bishop-Phelps-Bollob\'{a}s theorem, its applications, and relevant connections (also see the survey by Acosta \cite{Acosta surv}).

It is now natural to ask if a result like Bishop-Phelps-Bollob\'{a}s theorem holds for spaces of bounded linear operators between Banach spaces (however, see \cite{L} for counterexamples). Evidently, in order to answer such a question, one first needs to formulate an appropriate approximation problem for Banach spaces. Thanks to Acosta, Aron, Garc\'{i}a, and Maestre \cite[Definition 1.1]{acosta} for the following formulation of Bishop-Phelps-Bollob\'{a}s property for $\clb(X, Y)$, where $X$ and $Y$ are Banach spaces:

\begin{definition}[Acosta, Aron, Garc\'{i}a, and Maestre]\label{def: BPB space}
We say that $\clb(X,Y)$ (or simply $\clb(X)$ if $X = Y$) satisfies the Bishop-Phelps-Bollob\'{a}s property if for $\epsilon>0$ there exist $\beta(\epsilon) > 0$ and $\gamma(\epsilon) > 0$ with that $\lim_{t \raro 0} \beta(t)= 0$ such that for all $T\in S_{\clb(X,Y)}$, if $x \in S_X$ such that
\[
\|Tx\|> 1 - \gamma,
\]
then there exist $y \in S_X$ and $N \in S_{\mathscr{B}(X,Y)}$ such that
\begin{enumerate}
\item $\|N y\|=1$,
\item $\|x - y\|<\beta$, and
\item $\|T - N\|<\epsilon$.
\end{enumerate} 	
\end{definition}

There are many examples of $\clb(X,Y)$ that satisfy the Bishop-Phelps-Bollob\'{a}s property. For instance, Banach spaces $X$ and $Y$ satisfying property $\beta$ in the sense of Lindenstrauss \cite{acosta, L} (also see \cite{Aron1, Cho}), $L^1(\mu)$ spaces for $\sigma$-finite measures $\mu$, $C(K)$ spaces, finite-dimensional spaces \cite{acosta}, certain Banach spaces of holomorphic functions \cite{acosta 2, Kim}, spaces related to compact operators \cite{Dantas}, Asplund spaces \cite{Aron2, Cascales}, etc.

The main contribution of this paper is to demonstrate that the concept of Bishop-Phelps-Bollob\'{a}s property of simultaneous approximations (in the sense of Definition \ref{def: BPB space} and following the classical construction of Bishop, Phelps, and Bollob\'{a}s) fits perfectly in the framework of bounded analytic functions. This also includes specific values of the functions $\beta$ and $\gamma$ in Definition \ref{def: BPB space}. More specifically:

\begin{theorem}[Main result]\label{thm: Main intro}
$\clb(H^\infty)$ satisfies the Bishop-Phelps-Bollob\'{a}s property with $\beta(\epsilon) = \frac{2\epsilon}{7}$ and $\gamma(\epsilon) = \frac{1}{2}\left(\frac{\epsilon}{7}\right)^2$, $\epsilon > 0$.
\end{theorem}

The proof of the above result uses a Urysohn-type lemma for $H^\infty$ (see Lemma \ref{lem:1}), which we believe to be of independent interest. Although the statement of our lemma is conceptually similar to that of \cite[Lemma 2.5]{Cascales} (also see \cite[Lemma 3]{Kim}), the proof, as one would expect, uses more analytic machinery. Indeed, in the proof of this lemma (as well as the main theorem of this paper), we exploit the algebraic and analytic structure of the maximal ideal space of $H^\infty$. On the other hand, like \cite{Cascales, Kim}, we also use the idea of Stolz type regions. In Section \ref{sec: lemma}, we prove our variant of the Urysohn-type lemma for $H^\infty$, which we use to prove the main theorem in Section \ref{sec: main thm}.

Finally, in Section \ref{sec:final} we conclude this paper with the Bishop-Phelps-Bollob\'{a}s property for operator ideals of $\clb(H^\infty)$, outline some future directions of research on this topic, and some remarks concerning Asplund spaces.

\section{A Urysohn-type lemma for $H^\infty$}\label{sec: lemma}

The aim of this section is to prove a Urysohn-type lemma in the setting of bounded analytic functions. We start by recalling some of the basic facts from commutative Banach algebras \cite{Dales}.

Let $A$ be a commutative Banach algebra with unit. Let us denote by $\clm(A)$ the maximal ideal space of $A$. Note that $\clm(A)$ is the set of all nonzero  multiplicative linear functionals on $A$. Since $A$ is unital, it follows that
\[
\|\vp\|_{A^*} = 1 \qquad (\vp \in \clm(A)),
\]
that is, $\clm(A)$ is embedded into the unit sphere of $A^*$. Then, by the Banach–Alaoglu theorem, $\clm(A)$ is a compact
Hausdorff space in the weak-$*$ topology. If we denote by $C(\clm(A))$ the algebra of continuous functions on $\clm(A)$ with the supremum norm, then the \textit{Gelfand map} $\Gamma : A \raro C(\clm(A))$ is a contractive homomorphism of Banach algebras, where
\begin{align*}
\Gamma(f) = \hat{f},
\end{align*}
and
\[
\hat{f} (\vp) = \vp(f),
\]
for all $f\in A$ and $\vp \in \clm(A)$. Suppose, in addition, that $A$ is an algebra of complex-valued continuous functions on a compact Hausdorff space $\Omega$ that separates points of $\Omega$ and contains constant functions. Then a subset $C\subseteq \Omega$ is called a \textit{boundary} for $A$ if
\[
\underset{x\in \Omega}{\sup} |f(x)|=\underset{x\in C}{\max}|f(x)|,
\]
for all $f \in A$. The \textit{\v{S}ilov boundary} \cite[page 173]{Hoffman} for $A$, denoted by $\partial_S A$, is the smallest closed boundary for $A$, that is,
\[
\partial_S A = \bigcap \{C: C \text{ is a closed boundary for } A\} .
\]
In other words, $\partial_S A$ is the smallest closed subset of $\clm(A)$ on which every $\hat{f} \in \Gamma(A)$ attains its maximum modulus. Now we turn to the unital commutative Banach algebra $H^\infty$ \cite[Chapter 10]{Hoffman}. In this case
\begin{equation}\label{eqn: norm of infty}
\|f\|_\infty = \|\hat{f}\|_\infty = \sup_{\vp \in \clm(H^\infty)} |\hat{f}(\vp)|,
\end{equation}
for all $f \in H^\infty$, that is, $\Gamma$ an isometric isomorphism from $H^\infty$ into $C(\clm(H^\infty))$. Now we use the Gelfand map $\Gamma : H^\infty \raro C(\mathcal{M}(H^\infty))$ to identify $H^\infty$ with
\[
\widehat{H^\infty}:= \Gamma(H^\infty),
\]
Note that $\widehat{H^\infty}$ is a uniformly closed subalgebra of $C(\mathcal{M}(H^\infty))$ (that is, $\widehat{H^\infty}$ contains constant functions and separates the points \cite{Hoffman}).

Finally, recall that $H^\infty$ can be identified with a closed subalgebra of $L^\infty$ (via radial limits), where $L^\infty$ denotes the von Neumann algebra of all essentially bounded measurable complex-valued functions on the unit circle $\bt$. Denote by $\widetilde{H^\infty}$ the copy of $H^\infty$ in $L^\infty$. Recall that
\[
\widetilde{H^\infty}=L^{\infty} \cap \widetilde{H^2},
\]
where $\widetilde{H^2} \subseteq L^2$ is the Hardy space on the unit circle $\mathbb{T}$. In view of the above identification, we denote by $\tilde{\vp} \in \widetilde{H^\infty}$ the function corresponding to $\vp \in H^\infty$. This point of view is useful in identifying the \v{S}ilov boundary of $H^\infty$. More specifically (see \cite[page 174]{Hoffman}), the map
\[
\tau (\tilde{\vp}) = \tilde{\vp}|_{H^\infty}=\vp \qquad (\tilde{\vp} \in \clm(L^{\infty})),
\]
defines a homeomorphism $\tau:\clm(L^{\infty}) \rightarrow \clm(H^{\infty})$, and
\begin{equation}\label{eqn: tau M H}
\tau(\clm(L^{\infty}))=\partial_S H^{\infty}.
\end{equation}

Before proceeding to the Urysohn-type lemma, we record some observations: Let $f$ be an analytic function on $\bd$. We claim that $\|f\|_{\infty}<1$ if and only if $|\vp(f)|<1$ for all $\vp \in \clm(H^{\infty})$. Indeed, if $\|f\|_{\infty}<1$, then $\|\hat{f}\|_{\infty}=\|f\|_{\infty}<1$, and hence (or, see \eqref{eqn: norm of infty})
\[
\begin{split}
\|\hat{f}\|_{\infty}=\sup_{\vp \in \clm(H^{\infty})}|\hat{f}(\vp)|=\sup_{\vp \in \clm(H^{\infty})}|\vp(f)|,
\end{split}
\]
implies that $|\vp(f)|<1$ for all $\vp \in \clm(H^{\infty})$. To prove the converse direction, assume that $|\vp(f)|<1$ for all $\vp \in \clm(H^{\infty})$. Assume, if possible, that $\|f\|_{\infty}= 1$. Since $\sup_{z\in \mathbb{D}}|f(z)|=1$, there exists a sequence $\{z_n\}\subseteq \mathbb{D}$ such that $|f(z_n)| \rightarrow 1$. Let
\[
S := \{ev_{z_n}: z_n \in \mathbb{D}\}.
\]
Here, $ev_z \in \clm(H^{\infty})$, $z \in \bd$, is the evaluation map $ev_z(g) = g(z)$ for all $g \in H^\infty$. Clearly, $S \subseteq \clm(H^{\infty})$. Now, compactness of $\clm(H^{\infty})$ yields a limit point $\vp \in \clm(H^{\infty})$ of $S$. Therefore
\[
\displaystyle \vp(f)=\lim_{n \in \mathcal{U}} ev_{z_n}(f)=\lim_{n \in \mathcal{U}}f(z_n),
\]
where $\mathcal{U}$ is a free ultrafilter of $\mathbb{N}$. Hence, $\displaystyle|\vp(f)|=\lim_{n \in \mathcal{U}}|f(z_n)|=1$. This contradiction completes the proof of the claim.

The following assertion will be useful in what follows: Let $f$ be an analytic function on $\bd$ satisfying $|\vp(f)|<1$ for every $\vp\in \clm(H^{\infty})$ (equivalently, $\|f\|_\infty < 1$ by the above observation). Then
\begin{equation}\label{eqnpsi f phi}
(\psi \circ \hat{f})(\vp) = \widehat{\psi \circ f}(\vp),
\end{equation}
for all $\psi \in H^\infty$, and for all $\vp\in \clm(H^{\infty})$.  Since $\vp$ is multiplicative, we have $\vp(f^n) = \vp(f)^n$ for all $n \geq 0$. Let $\psi \in H^\infty$, and suppose
\[
\psi(z)= \sum_{n=0}^\infty a_n z^n \qquad (z \in \mathbb{D}).
\]
With these assumptions, the series $\sum_{n=0}^{\infty}a_nf^n$ converges in the norm $\|\cdot\|_{\infty}$. Then
\begin{align*}
(\psi \circ \hat{f})(\vp)& = \psi (\vp(f))
\\
& = \underset{n=0}{\overset{\infty}{\sum}}a_n(\vp (f))^n
\\
& = \underset{n=0}{\overset{\infty}{\sum}}a_n\vp (f^n)\\
& = \vp \Big(\underset{n=0}{\overset{\infty}{\sum}}a_nf^n\Big)
\\
& = \vp (\psi \circ f)
\\
& = \widehat{\psi \circ f}(\vp),
\end{align*}
completes the verification of \eqref{eqnpsi f phi}. With this observation and preliminaries in place, we are ready for the Urysohn-type lemma. As noted before, our analytic variant of the Urysohn-type lemma seems to be of independent interest.

\begin{lemma}\label{lem:1}
Let $U$ be an open subset of $\clm(H^{\infty})$. Then for each $\epsilon \in (0,1)$ and $\vp_0\in U \cap \partial_S H^{\infty}$, there exists $\hat{\psi}\in\widehat{H^{\infty}}$ such that
\begin{enumerate}
\item $\|\hat{\psi}\| =\hat{\psi}(\vp_0) =1$,
\item $\sup\{|\hat{\psi}(\vp)|: \vp\in \partial_S H^{\infty} \setminus U\} < \epsilon$, and
\item $|\hat{\psi}(\vp)|+(1-\epsilon)|1-\hat{\psi}(\vp)|\leq 1$ for all $\vp \in \clm(H^{\infty})$.
\end{enumerate}
\end{lemma}
	
\begin{proof}
We consider the Stolz region $\Omega_{\epsilon}$ defined by
\[
\Omega_{\epsilon}: = \{z\in\bc:|z|+(1-\epsilon)|1-z|\leq 1\}.
\]
Note that $0$ and $1$ are in $\Omega_{\epsilon}$, and there exists a homeomorphism $\psi_{\epsilon}: \overline{\bd}\rightarrow \Omega_{\epsilon}$ \cite[Theorem 14.8, 14.19]{Rudin} such that
\begin{enumerate}
\item $\psi_{\epsilon}|_{\bd}$ is a conformal mapping onto the interior of $\Omega_{\epsilon}$,
\item $\psi_{\epsilon}(1)=1$, and
\item $\psi_{\epsilon}(0)=0$.
\end{enumerate}
Also observe that $\epsilon^2 \overline{\bd} \subseteq \Omega_\epsilon$. Since $0 \in {\psi_\epsilon}^{-1}(\epsilon^2 \bd)$ and ${\psi_\epsilon}^{-1}(\epsilon^2 \bd)$ is an open set, there exists $\delta>0$ such that $\delta<\epsilon$ and
\begin{align}\label{Eq:1}
\delta \bd \subseteq {\psi_\epsilon}^{-1}(\epsilon^2 \bd).
\end{align}
Choose $\epsilon'>0$ such that $0<\epsilon'<\delta<\epsilon$. We claim that there exist $f_1\in H^{\infty}(\mathbb{D})$ such that $|\vp_0(f_1)|=1$ and
\[
\hat{f}_1(\partial_SH^{\infty}\setminus U)\subseteq\epsilon'\bd.
\]
We know that $\partial_SH^{\infty}$ is a compact Hausdorff totally disconnected topological space (in fact, \cite[Theorem 3.2]{GAMELIN} implies that $\partial_SH^{\infty}$ is extremely disconnected), and hence the topology of $\partial_SH^{\infty}$ is generated by clopen sets. Consequently, there exists a clopen set $V\subseteq U\cap\partial_SH^{\infty}$ such that $\vp_0\in V$. In view of \eqref{eqn: tau M H} (or the theorem in \cite[Page 174]{Hoffman}), the map
\[
\tau(\tilde{\vp}) = \tilde{\vp}|_{H^{\infty}}=\vp \qquad (\tilde{\vp} \in \clm(L^\infty)),
\]
defines an onto homeomorphism $\tau:\clm(L^\infty)\rightarrow\partial_SH^{\infty}\subseteq\clm(H^{\infty})$. Therefore, $\tau^{-1}(V)$ is clopen in $\clm(L^\infty)$, and hence there exists a measurable set $E\subseteq \bt$ (cf. the corollary in \cite[Page 170]{Hoffman}) such that
\[
\tau^{-1}(V):=\{{\tilde{\vp}}\in\clm(L^{\infty}):\hat{\chi}_E({\tilde{\vp}})=1\}.
\]
Now we proceed as in the proof of the theorem in \cite[Page 174]{Hoffman}. Let $u$ be the harmonic extension of $\chi_E$ to $\bd$ and $v$ be the harmonic conjugate of $u$. Set
\[
f=e^{u-1+iv}\in H^{\infty}.
\]
Identifying $f$ with $F\in L^{\infty}$, it follows that
\[
|\hat{F}({\tilde{\vp}})|=e^{\hat{\chi}_E({\tilde{\vp}})-1} \qquad ({\tilde{\vp}}\in\clm(L^{\infty})).
\]
In other words
\begin{equation*}
|\hat{F}({\tilde{\vp}})|
=\begin{cases} 1&\text{ if }\tilde{\vp}\in\tau^{-1}(V)
\\
1/e&\text{ if }\tilde{\vp}\in \clm(L^{\infty})\setminus\tau^{-1}(V).
\end{cases}
\end{equation*}
By using the identification $H^{\infty}$ with $L^{\infty} \cap \widetilde{{H}^2}$, we have
\[
\begin{split}
\hat{F}\circ\tau^{-1}(\vp)& =\tau^{-1}(\vp)(F)
\\
&=\tilde{\vp}(F)
\\
& = \vp(f)
\\
& =\hat{f}(\vp),
\end{split}
\]
for all $\vp\in\partial_SH^{\infty}$. Hence
\begin{equation*}
|\hat{f}(\vp)| = \begin{cases}
1&\text{ if }\vp\in V\\
1/e&\text{ if }\vp\in\partial_SH^{\infty}\setminus V.
\end{cases}
\end{equation*}
Clearly, for $\epsilon'>0$, there exists $n_0\in\mathbb{N}$ such that
\[
e^{-n}<\epsilon' \qquad (n\geq n_0).
\]
Define
\[
f_1=f^{n_0}.
\]
Then
\[
f_1(\partial_S\setminus U)\subseteq\epsilon'\bd,
\]
and
\[
|\vp_0(f_1)|=1.
\]
If $\vp_0(f_1)=\alpha$ for some $\alpha\in\mathbb{T}$, then consider $f_2=\bar{\alpha}f_1$. Therefore, $\vp_0(f_2)=1$, $\|f_2\|=\|\hat{f}_2\|=1$, and
\begin{equation}\label{eqn 2}
\hat{f}_2(\partial_SH^{\infty}\setminus U)\subseteq\epsilon'\bd.
\end{equation}
We now claim that $\hat{\psi} \in\widehat{H^{\infty}}$, where
\[
\hat{\psi}:= \psi_{\epsilon} \circ\hat{f_2}.
\]
Since $\|\hat{f_2}\|=1$, we have $|\vp(f_2)| \leq 1$ for all $\vp \in \clm(H^{\infty})$.
Let $A \subseteq \clm(H^\infty)$ such that
\[
|\eta(f_2)| = 1 \qquad (\eta\in A).
\]
For each $n \geq 1$, define
\[
h_n = \frac{n}{n+1}f_2.
\]
Then
\[
\begin{split}
|\eta( h_n)| & = \left|\frac{n}{n+1} \eta(f_2) \right|
\\
& = \frac{n}{n+1}
\\
& <1,
\end{split}
\]
for all $\eta\in A \subseteq \clm(H^\infty)$, and
\[
|\vp(h_n)|<1 \qquad (\vp \in \clm(H^\infty)\setminus A).
\]
Hence
\[
(\psi_{\epsilon} \circ \hat{h}_n)(\vp)=\widehat{(\psi_{\epsilon} \circ h_n)}(\vp) \qquad (n \geq 1).
\]
Since $\psi_{\epsilon}$ is uniform continuous on $\overline{\bd}$, for every $\xi>0$, there exists $\gamma>0$ such that
\[
|\psi_{\epsilon}(z)-\psi_{\epsilon}(w)| < \frac{\xi}{2} \qquad (z,w\in\overline{\bd}, |z-w|<\gamma).
\]	
Pick a natural number $N$ such that $\frac{1}{n+1}<\gamma$ for all $n\geq N$. Then, for each $n \geq N$, we have
\begin{align*}
|(\psi_{\epsilon} \circ \hat{f}_2)(\vp) - \widehat{\psi_{\epsilon}\circ f_2}(\vp)| & \leq  |(\psi_{\epsilon} \circ \hat{f}_2)(\vp) - (\psi_{\epsilon} \circ \hat{h}_n)(\vp)| + |(\psi_{\epsilon} \circ \hat{h}_n)(\vp) - \widehat{\psi_{\epsilon} \circ f_2}(\vp)|
\\
& = \left|\psi_{\epsilon} (\vp (f_2))-\psi_{\epsilon} \left(\frac{n}{n+1}\vp (f_2)\right)\right| + |(\widehat{\psi_{\epsilon} \circ h_n} - \widehat{\psi_{\epsilon} \circ f_2})(\vp)|
\\
& < \frac{\xi}{2} + |(\widehat{\psi_{\epsilon} \circ h_n} - \widehat{\psi_{\epsilon} \circ f_2})(\vp)|
\\
& \leq \frac{\xi}{2}+\|\psi_{\epsilon} \circ h_n - \psi_{\epsilon} \circ f_2\|\|\vp\|,
\end{align*}
where on the other hand
\begin{align*}
\|\psi_{\epsilon} \circ h_n - \psi_{\epsilon} \circ f_2\| & = \underset{z\in\bd}{\sup}|\psi_{\epsilon}( h_nz)-\psi_{\epsilon}(f_2z)|
\\
& = \underset{z\in\bd}{\sup}\left|\psi_{\epsilon} \left( \frac{n}{n+1}f_2(z)\right)-\psi_{\epsilon}(f_2(z))\right|\\
& < \frac{\xi}{2}.
\end{align*}
This finally implies that
\[
|(\psi_{\epsilon} \circ  \hat{f_2})(\vp) - \widehat{(\psi_{\epsilon} \circ  f_2)}(\vp)|<\xi.
\]
Thus we obtain
\[
(\psi_{\epsilon} \circ  \hat{f_2})(\vp) = \widehat{(\psi_{\epsilon} \circ f_2)}(\vp) \qquad (\vp\in\clm(H^{\infty})),
\]
and hence
\[
\hat \psi = \psi_{\epsilon} \circ  \hat{f_2} = \widehat{\psi_{\epsilon} \circ  f_2} \in \widehat{H^{\infty}}.
\]
This proves the claim. Now, since $\vp_0(f_2) =1$ and $\psi_{\epsilon}(1) = 1$, it follows that
\[
\hat{\psi}(\vp_0)=(\psi_{\epsilon} \circ  \hat{f_2})( \vp_0) = \psi_{\epsilon} (\vp_0( f_2)) = \psi_{\epsilon}(1)=1.
\]
Moreover
\[
\|\hat{\psi}\|=\|\psi_{\epsilon} \circ \hat{f_2}\| \leq \|\psi_{\epsilon}\|\|\hat{f_2}\|=1,
\]
and
\[
1=\hat{\psi}(\vp_0)\leq\|\hat{\psi}\|,
\]
implies that $\|\hat{\psi}\|= \hat{\psi}(\vp_0) = 1$. Finally, by \eqref{Eq:1} and \eqref{eqn 2}, we have
\[
\psi_{\epsilon}(\delta \bd) \subseteq \epsilon^2\bd,
\]
and
\[
\psi_\epsilon(\hat{f_2}(\partial_S H^{\infty}\setminus U)) \subseteq \psi_\epsilon(\epsilon{'}\bd),
\]
respectively, and hence
\begin{align*}
\hat{\psi}\left( \partial_S H^{\infty}\setminus U\right)& = \psi_\epsilon(\hat{f_2}(\partial_S H^{\infty}\setminus U))
\\
&\subseteq \psi_\epsilon(\epsilon{'}\bd)
\\
&\subseteq \psi_{\epsilon}(\delta \bd)
\\
& \subseteq \epsilon^2\bd
\\
&\subseteq \epsilon \bd.
\end{align*}
Therefore, if $\vp\in \clm(H^\infty)$, then the estimate
\begin{align*}
|\hat{\psi}(\vp)|+(1-\epsilon)|1-\hat{\psi}(\vp)| & = |\psi_\epsilon(\vp(f_2))|+(1-\epsilon)|1-\psi_\epsilon(\vp(f_2))|
\\
& \leq 1,
\end{align*}
completes the proof of the lemma.
\end{proof}

The statement and part of the proof of Lemma \ref{lem:1} are motivated by Urysohn-type lemma of \cite[Lemma 2.5]{Cascales}  (also see \cite[Lemma 3]{Kim}). For instance, the uses of Stolz domains follow the constructions of \cite{Cascales, Kim}. Also, see the proof of Bollob\'{a}s \cite{Bollobas} for a similar (but not exactly the same) construction in the setting of $\mathbb{R}^2$. However, as one would expect, the algebraic and analytic tools of the maximal ideal space of $H^\infty$ play a key role in the present consideration.

\section{Proof of the main theorem}\label{sec: main thm}

The purpose of this section is to present the proof of Theorem \ref{thm: Main intro}. In addition to Lemma \ref{lem:1}, the proof also needs one more lemma. We begin with some basic definitions and results from topology.

A topological space $X$ is \textit{extremally disconnected} if the closure of every open set in $X$ is open. Also recall that extremally disconnected, compact and Hausdorff spaces are called \textit{Stonean} spaces. The Stone-\v{C}ech compactification of a discrete space is a typical example of Stonean spaces. A compact and Hausdorff space $X$ is Stonean if and only if $X$ is a retract of the Stone-\v{C}ech compactification of a discrete space (cf. \cite[Theorem 24.7.1]{Semadeni}).

\begin{lemma}\label{Lemma isolated point silov bd}
Let $f\in H^{\infty}(\bd)$, $\vp_0\in\partial_S H^{\infty}$, and let $0<\epsilon<1$. Define the sub-basic open subset $U$ of $\partial_S H^{\infty}$ around $\vp_0$ by
\[
U=\{\vp\in\partial_S H^{\infty}: |\vp(f)-\vp_0(f)|<\epsilon\}.
\]
Then $\overline{U}$ contains an isolated point of $\partial_S H^{\infty}$.
\end{lemma}
\begin{proof}
Since $\partial_SH^{\infty}$ is a closed subset of $\clm(H^{\infty})$, it follows that $\partial_SH^{\infty}$ is a compact and Hausdorff space. %We also know \cite[p. 169]{Hoffman} that the space $\clm(L^{\infty})$ is a extremally disconnected.
By %\eqref{eqn: tau M H},
\cite{GAMELIN} we know that $\partial_S H^{\infty}$ is an extremally disconnected space. Therefore, $\partial_SH^{\infty}$ is a Stonean space.

\noindent Observe that
\[
\overline{U}=\{\vp\in\partial_S H^{\infty}: |\vp(f)-\vp_0(f)|\leq\epsilon\}.
\]
If $\overline{U}$ is a finite set, then using the Hausdorff property of $\partial_SH^{\infty}$, we get that $\overline{U}$ contains an isolated point. Next we assume that $\overline{U}$ is an infinite set. By the discussion preceding the statement of this lemma, $\partial_SH^{\infty}$ is a retract of $\beta X$ for some discrete space $X$. Here $\overline{U}$ is a clopen set in $\partial_SH^{\infty}$. Since $\beta X$ is an extremally disconneted space, the topology of $\beta X$ is generated by clopen sets \cite[Theorem 3.18]{HS}. Hence, there exists a clopen set $V$ in $\beta X$ such that
\[
\overline{U}=V \cap  \partial_SH^{\infty},
\]
and we conclude that $\overline{U}$ is an infinite closed subset of $\beta X$. Therefore
\[
\overline{U} \subseteq \partial_SH^{\infty} \subseteq \beta X.
\]
Now, by \cite[Lemma 4]{KIM} (or \cite[p. 137]{GJ}), there exists a subspace $Y \subseteq \overline{U}$ such that
\[
Y= F(\beta \mathbb{N}),
\]
for some homeomorphism $F$. Consequently
\[
Y \subseteq \overline{U} \subseteq \partial_S H^{\infty}\subseteq \beta X.
\]
Consider inclusion maps $i_1: Y \longrightarrow \overline{U}$, $i_2:\overline{U} \longrightarrow \partial_SH^{\infty}$, and $i_3:\partial_SH^{\infty} \longrightarrow \beta X$. We have the map
\[
\tilde{F}:= i_3 \circ i_2 \circ i_1 \circ F : \beta \mathbb{N} \longrightarrow \beta X.
\]
Since $\tilde{F}$ is continuous, it follows that $Y$ is compact in $\beta X$, which implies, in particular, that $Y$ is closed in $\beta X$. Thus we conclude that $Y\cong \beta \mathbb{N}$ is identical with the closure in $\beta X$ of $\mathbb{N}$. Consequently (cf. \cite[p. 89]{GJ} or \cite[page 37]{KIM}), we have that $\mathbb{N}\subseteq X$ is $C^*$-embedded. Since every point of $X$ is isolated in $\beta X$, we conclude that every point of $\mathbb{N}$ is isolated in $\beta X$. In summary, we have the following:
\[
\mathbb{N} \subseteq Y \cong \beta\mathbb{N} \subseteq \overline{U} \subseteq \partial_SH^{\infty}\subseteq \beta X.
\]
It is now clear that every point of $\mathbb{N}$ is isolated in $\partial_SH^{\infty}$. This implies that $\overline{U}$ contains an isolated point of $\partial_SH^{\infty}$, and completes the proof of the lemma.
\end{proof}

We fix some more notation before proceeding to the main result. Recall that the Gelfand map $\Gamma : H^\infty \raro C(\clm(H^\infty))$ is an isometry, and $\widehat{H^\infty} := \Gamma(H^\infty)$. For each $X\in \mathscr{B}( H^\infty)$, define $\tilde{X} : \widehat{H^\infty} \raro \widehat{H^\infty}$ by
\[
\tilde{X}= \Gamma X \Gamma^{-1}.
\]
Therefore, $\tilde{X}\widehat{f} = \widehat{X f}$ for all $\widehat{f}\in \widehat{H^\infty}$. Moreover, since $\Gamma$ is an isometry, it follows that
\[
\|\tilde{X} \widehat{f}\|=\|\Gamma X \Gamma^{-1} (\widehat{f})\|=\|X\Gamma^{-1}f\|=\|Xf\| \qquad (f\in H^\infty),
\]
and
\[
\|\tilde{X}\|= \underset{\widehat{f}\in \widehat{H^\infty}, \|\widehat{f}\|=1} {\sup}\|\tilde{X}{\widehat{f}}\| = \underset{f\in H^\infty, \|f\|=1}{\sup}\|X f\|=\|X\|.
\]
Now we are ready to prove the main theorem of this paper. However, for the reader's convenience, we restate Theorem \ref{thm: Main intro} by incorporating the definition of Bishop-Phelps-Bollob\'{a}s property (see Definition \ref{def: BPB space}).

\begin{theorem}\label{thm:main}
Let $f_0$ be a unit vector in $H^{\infty}$, $T\in\mathscr{B}(H^{\infty})$, and let $\epsilon \in (0, 1)$. Suppose $\|T\|=1$ and
\[
\|Tf_0\|>1-\frac{\epsilon}{2}.
\]
Then there exist $N \in \clb(H^{\infty})$ and a unit vector $g_0\in H^{\infty}$ such that
\begin{enumerate}
\item $\|N g_0\|=\|N\|=1$,
\item $\|f_0 - g_0\|<2 \sqrt{\epsilon}$, and
\item $\|T - N\|<7\sqrt{\epsilon}$.
\end{enumerate}
\end{theorem}
\begin{proof}
By the assumption that $\|Tf_0\|>1-\frac{\epsilon}{2}$, it follows that
\[
\|\widehat{Tf_0}\|=\|Tf_0\| > 1- \frac{\epsilon}{2}.
\]
Then there exists $\vp_0 \in \partial_S H^\infty \subseteq \mathcal{M}(H^\infty)$ (recall that $\partial_S H^\infty $ is the \v{S}ilov boundary for $H^\infty$) such that
\begin{align*}
|\widehat{Tf_0}(\vp_0)| = \|\widehat{Tf_0}\|> 1- \frac{\epsilon}{2}.
\end{align*}
Since $\vp_0\in \mathcal{M}(H^\infty)$, it follows that $\|\vp_0\|=1$, and hence $\|\vp_0 \circ T \| \leq 1$ and
\[
|\vp_0(T(f_0))| = |(\vp_0 \circ T) f_0| = |\widehat{Tf_0}(\vp_0)| > 1 - \frac{\epsilon}{2}.
\]
We claim that there exist an isolated point $\tilde{\vp}_0\in\partial_SH^{\infty}$ such that
\[
|\tilde{\vp}_0(Tf_0)|>1-\epsilon.
\]
If $\vp_0$ itself is an isolated point of $\partial_SH^{\infty}$, then we can choose $\tilde{\vp}_0=\vp_0$.  Suppose $\vp_0$ is not an isolated point of $\partial_SH^{\infty}$. Since $\partial_SH^{\infty}$ is extremally disconnected, it follows that
\[
W:=\left\{\vp\in\partial_SH^{\infty}:|\vp(Tf_0)-\vp_0(Tf_0)|\leq\frac{\epsilon}{4}\right\}.
\]
is a clopen set. By Lemma \ref{Lemma isolated point silov bd}, there exists $\tilde{\vp}_0 \in W$ such that $\tilde{\vp}_0$ is an isolated point of $\partial_SH^{\infty}$. Then
\begin{align*}
|\tilde{\vp}_0(Tf_0)| \geq & |\vp_0(Tf_0)|-|\vp_0(Tf_0)-\tilde{\vp}_0(Tf_0)|
\\
> & 1-\frac{\epsilon}{2}-\frac{\epsilon}{4}
\\
> & 1-\epsilon.
\end{align*}
Then applying the Bishop-Phelps-Bollob\'{a}s theorem (cf. Theorem \ref{BPB}) to the linear functional
\[
\dfrac{1}{\|\tilde{\vp}_0 \circ T\|} \tilde{\vp}_0 \circ T : H^\infty \raro \mathbb{C},
\]
we immediately get a linear functional $\vp_1: H^\infty \raro \mathbb{C}$, $\|\vp_1\|=1$, and a unit vector $g_0\in H^\infty$ (that is, $\|g_0\|_\infty =1$) such that
\begin{enumerate}[(i)]
\item $|\vp_1(g_0)|=1$,
\item $\|g_0-f_0\|\leq \sqrt{2\epsilon} < 2\sqrt{\epsilon} $, and
\item $\left \| \vp_1 -\dfrac{\tilde{\vp}_0 \circ T}{\|\tilde{\vp}_0 \circ T\|}\right \| \leq \sqrt{2\epsilon} < 2\sqrt{\epsilon}$.
\end{enumerate}
As we already know that $\tilde{\vp}_0$ is an isolated point of $\partial_SH^{\infty}$, there exists an open set $U\subseteq\clm(H^{\infty})$ such that
\[
U\cap\partial_SH^{\infty} = \{\tilde{\vp}_0\}.
\]
Since $\tilde{\vp}_0\in U \cap \partial_S H^\infty$, by Lemma \ref{lem:1} there exists $\hat{\psi}\in\widehat{H^{\infty}}$ such that
\begin{enumerate}[(a)]
\item $\|\hat{\psi}\|=1=\hat{\psi}(\tilde{\vp}_0)$,
\item $\sup\{|\hat{\psi}(\vp)|: \vp\in{\partial_SH^{\infty}\setminus U}\} < \epsilon$, and
\item $|\hat{\psi}(\vp)|+(1-\epsilon)|1-\hat{\psi}(\vp)|\leq 1$ for all $\vp\in\clm(H^{\infty})$.
\end{enumerate}
With (c) in mind, we introduce a bounded linear operator $\tilde{N}: \widehat{H^{\infty}} \raro \widehat{H^{\infty}}$ defined by
\[
(\tilde{N}\hat{f}) (\vp) = \widehat{\psi}(\vp) \vp_1(f) + (1-\epsilon) (1- \widehat{\psi}(\vp)) (\tilde{T} \widehat{f})(\vp),
\]
for all $\widehat{f}\in \widehat{H^{\infty}}$ and $\vp \in \clm(H^{\infty})$. We claim that $\|\tilde{N}\| = 1$. Indeed, for each $\widehat{f}\in \widehat{H^{\infty}}$ with $\|\widehat{f}\|\leq 1$ and $\vp\in \clm(H^{\infty})$ we have
\begin{align*}
\left|(\tilde{N} \hat{f}) (\vp)\right| & \leq |\widehat{\psi}(\vp)| |\vp_1(f)| +(1-\epsilon) |1- \widehat{\psi}(\vp)| |(\tilde{T} \widehat{f})(\vp)|
\\
& \leq |\widehat{\psi}(\vp)|+(1-\epsilon)|1-\widehat{\psi}(\vp)|
\\
& \leq 1,
\end{align*}
where the last inequality follows from (c) above. This implies, of course, that $\|\tilde{N}\| \leq  1$. On the other hand, by (i) and (a), we have
\[
|\tilde{N}(\widehat{g_0})(\tilde{\vp}_0)|=|\vp_1(g_0)| |\widehat{\psi}(\tilde{\vp}_0)| + 0 =1.
\]
This, together with $\|\tilde{N}\| \leq 1$, implies that $\|\tilde{N}\| = 1$ and proves the claim. We now estimate $\|\tilde{T} - \tilde{N}\|$. For each $\hat{f}\in \widehat{H^{\infty}}$ with $\|\widehat{f}\| \leq 1$, we have
\begin{align*}
\|\tilde{T} \widehat{f} - \tilde{N} \widehat{f}\| & = \underset{\vp \in \partial_SH^{\infty}}{\sup} \left|(\tilde{T} \widehat{f}) (\vp)-\vp_1(f)\widehat{\psi}(\vp) -(1-\epsilon) (1- \widehat{\psi}(\vp)) (\tilde{T}\widehat{f})(\vp) \right|
\\
& = \underset{\vp \in \partial_SH^{\infty}}{\sup} \left|\epsilon (1- \widehat{\psi}(\vp)) (\tilde{T}\widehat{f})(\vp) + \widehat{\psi}(\vp) (\tilde{T} \widehat{f})(\vp) -\vp_1(f)\widehat{\psi}(\vp) \right|
\\
& \leq  \epsilon \; \underset{\vp \in \partial_SH^{\infty}}{\sup} |1- \widehat{\psi}(\vp)| |(\tilde{T} \widehat{f})(\vp)| + \underset{\vp \in \partial_SH^{\infty}}{\sup} |\widehat{\psi}(\vp)(\vp_1(f) - (\tilde{T} \widehat{f})(\vp))|
\\
& \leq 2 \epsilon + \underset{\vp \in \partial_SH^{\infty}}{\sup} |\widehat{\psi}(\vp)| |\vp_1(f) - (\tilde{T} \widehat{f})(\vp)|,
\end{align*}
as $\|\widehat{\psi}\| = 1$ (see (a)) and $\|\tilde{T} \widehat{f}\| \leq 1$. We rewrite $\partial_SH^{\infty}$ as
\[
\partial_S H^{\infty} = ((\clm(H^\infty) \setminus U) \cap \partial_SH^{\infty}) \cup (U \cap \partial_SH^{\infty}).
\]
In view of this partition, we have
\[
\begin{split}
\underset{\vp \in \partial_S H^{\infty}}{\sup} |\widehat{\psi}(\vp)| |\vp_1(f) - (\tilde{T} \widehat{f})(\vp)| & \leq \underset{\vp \in (\clm(H^\infty) \setminus U) \cap \partial_S H^{\infty}}{\sup} |\widehat{\psi}(\vp)| |\vp_1(f) - (\tilde{T} \widehat{f})(\vp)|
\\
& \qquad + \underset{\vp \in U \cap \partial_S H^{\infty}}{\sup} |\widehat{\psi}(\vp)| |\vp_1(f) - (\tilde{T} \widehat{f})(\vp)|.
\end{split}
\]
We further estimate the second term of the right-hand side, say $\alpha$, as follows:
\begin{align*}
\alpha & = \underset{\vp \in U \cap \partial_S H^{\infty}}{\sup} |\widehat{\psi}(\vp)| |\vp_1(f) - (\tilde{T} \widehat{f})(\vp)|
\\
& =  |\widehat{\psi}(\tilde{\vp}_0)| |\vp_1(f)- \widehat{T(f)}(\tilde{\vp}_0)|
\\
& =  |\widehat{\psi}(\tilde{\vp}_0)| |\vp_1(f)- \tilde{\vp}_0(Tf)|
\\
& \leq  \left| \vp_1(f)-
\dfrac{(\tilde{\vp}_0 \circ T)(f)}{\|\tilde{\vp}_0 \circ T\|} \right| + \left| \dfrac{(\tilde{\vp}_0 \circ T)(f)}{\|\tilde{\vp}_0 \circ T\|} - (\tilde{\vp}_0 \circ T)(f) \right|.
\end{align*}
By (iii), we have
\begin{align*}
\alpha & < 2\sqrt{\epsilon} \|f\| + \left|1-\|\tilde{\vp}_0 \circ T\| \right| \|f\| 	\\
& < 2\sqrt{\epsilon} + \epsilon
\\
& \leq 3\sqrt{\epsilon},
\end{align*}
where the last inequality follows from the fact that
\[
\|\tilde{\vp}_0 \circ T\| \geq |(\tilde{\vp}_0 \circ T)f_0| > 1 - \epsilon.
\]
Finally, (b) implies that
\[
\underset{\vp \in (\clm(H^\infty) \setminus U)\cap \partial_S H^{\infty}}{\sup} |\widehat{\psi}(\vp)| < \epsilon,
\]
and hence
\begin{align*}
\underset{\vp \in (\clm(H^\infty) \setminus U)\cap \partial_S H^{\infty}}{\sup} |\widehat{\psi}(\vp)| \left|\vp_1(f) - (\tilde{T} \widehat{f})(\vp) \right| & \leq  \underset{\vp \in (\clm(H^\infty) \setminus U)\cap \partial_S H^{\infty}} {\sup} \epsilon  \left|\vp_1(f) - (\tilde{T} \widehat{f})(\vp) \right|
\\
& \leq 2\epsilon.
\end{align*}
This, combined with the above estimates, then implies that $\|\tilde{T}-\tilde{N}\| < 7 \sqrt{\epsilon}$. Then $N: H^\infty \raro H^\infty$ defined by $N= \Gamma^{-1} \tilde{N} \Gamma$ satisfies all the required conclusions of the theorem.
\end{proof}

As already pointed out in Theorem \ref{thm: Main intro}, the above theorem also yields the value of $\beta$ and $\gamma$ in Definition \ref{def: BPB space} as $\beta(\epsilon) = \frac{2\epsilon}{7}$ and $\gamma(\epsilon) = \frac{1}{2}\left(\frac{\epsilon}{7}\right)^2$, $\epsilon > 0$.

\section{Concluding remarks}\label{sec:final}

In this section, we present an application of our approach to the Bishop-Phelps-Bollob\'{a}s property of operator ideals of $\clb(H^\infty)$ and make a couple of remarks. We begin with operator ideals of $\clb(H^\infty)$. Recall that a subset $\mathscr{I}(X) \subseteq \mathscr{B}(X)$ is called an \textit{operator ideal} if $\mathscr{I}(X)$ contains all finite rank operators and
\[
T_1\circ T \circ T_2 \in \mathscr{I}(X),
\]
for all $T \in \mathscr{I}(X)$ and $T_1,T_2 \in \mathscr{B}(X)$.

\begin{corollary}
Let $\mathscr{I}(H^\infty)$ be an operator ideal. Let $f_0$ be a unit vector in $H^{\infty}$, $T\in \mathscr{I}(H^\infty)$, and let $\epsilon \in (0, 1)$. Suppose $\|T\|=1$ and
\[
\|Tf_0\|>1-\frac{\epsilon}{2}.
\]
Then there exist $N \in \mathscr{I}(H^\infty)$ and a unit vector $g_0\in H^{\infty}$ such that
\begin{enumerate}
\item $\|N g_0\|=\|N\|=1$,
\item $\|f_0 - g_0\|<2\sqrt{\epsilon}$, and
\item $\|T - N\|<7\sqrt{\epsilon}$.
\end{enumerate}
\end{corollary}

\begin{proof}
Define $\mathscr{I}_{\infty}:=\{ \Gamma W \Gamma^{-1}|~W \in \mathscr{I}(H^\infty)\}$. Clearly $\mathscr{I}_{\infty}$ contains all finite rank operators. Let $\tilde{W}_1$ and $\tilde{W}_2$ are in $\mathscr{B}(\widehat{H^\infty})$ and $W \in \clb(H^\infty)$. Then
\[
\tilde{W}_1 \Gamma W \Gamma^{-1} \tilde{W}_2 =\Gamma W_1 \Gamma^{-1}\Gamma W \Gamma^{-1}\Gamma W_2 \Gamma^{-1}=\Gamma W_1 W W_2 \Gamma^{-1} \in \mathscr{I}_{\infty},
\]
implies that $\mathscr{I}_{\infty}$ is an operator ideal of $\mathscr{B}(\widehat{H^\infty})$. In particular, $\tilde{T} \in \mathscr{I}_{\infty}$.
Now, if we view $T$ as a bounded linear operator on $H^\infty$, then applying Theorem \ref{thm:main} to $T \in \mathscr{I}(H^\infty)$, we see that the corollary follows except for the fact that $N \in \mathscr{I}(H^{\infty})$. It is therefore enough to prove that $N \in \mathscr{I}(H^{\infty})$, where $N \in \clb(H^\infty)$ corresponds the operator $\tilde{N}: \widehat{H^{\infty}} \raro \widehat{H^{\infty}}$ defined by (see the proof of Theorem \ref{thm:main})
\[
(\tilde{N}\hat{f}) (\vp) = \vp_1(f) \widehat{\psi}(\vp) + (1-\epsilon) (1- \widehat{\psi}(\vp)) (\tilde{T} \widehat{f})(\vp),
\]
for all $\hat{f} \in \widehat{H^\infty}$ and $\vp \in \clm(H^\infty)$. Since $N = \Gamma^{-1} \tilde{N} \Gamma$, we only need to prove that $\tilde{N} \in \mathscr{I}_{\infty}$. To this end, we write $\tilde{N} = \tilde{N}_1 + \tilde{N}_2$, where
\[
\tilde{N}_1 \widehat{f} = \vp_1(f) \widehat{\psi} \text{ and } \tilde{N}_2 \widehat{f} = (1-\epsilon) (1- \widehat{\psi}) \widehat{T}\widehat{f},
\]
for all $\hat{f} \in \widehat{H^\infty}$. In view of the fact that $\mathscr{I}_{\infty}$ is an ideal, we simply prove that $\tilde{N}_1$ and $\tilde{N}_2$ are in $\mathscr{I}_{\infty}$. Evidently, $\tilde{N}_1$ is a rank one operator, and hence $\tilde{N}_1 \in \mathscr{I}_{\infty}$. On the other hand, $\tilde{T} \in \mathscr{I}_{\infty}$ implies that $\tilde{N} \in \mathscr{I}_{\infty}$, which completes the proof of the corollary.
\end{proof}

In particular, this applies equally to common operator ideals, like the ideal of weakly compact operators and the ideal of $p$-summing operators of $\clb(H^\infty)$. As an example, let us state the weakly compact operator ideal version of the above corollary.

Given a Banach space $X$, we denote the closed unit ball $\{x \in X: \|x\| \leq 1\}$ as $B_X$. Recall that a bounded linear operator $T \in \mathscr{B}(X)$ is called \textit{weakly compact} if $T(B_X)$ is relatively weakly compact. We denote by $\mathscr{W}(X)$ the space of all weakly compact operators on $X$. The following result is now a particular case of the above corollary. Note that the present formulation is precisely the weakly compact variant of Bishop-Phelps-Bollob\'{a}s property for compact operators analyzed by Dantas, Garc\'{i}a, Maestre and Mart\'{i}n in \cite{Dantas}.

\begin{corollary}\label{Cor:1}
Let $f_0$ be a unit vector in $H^{\infty}$, $T\in \mathscr{W}(H^{\infty})$, and let $\epsilon \in (0, 1)$. Suppose $\|T\|=1$ and
\[
\|Tf_0\|>1-\frac{\epsilon}{2}.
\]
Then there exist $N \in \mathscr{W}(H^{\infty})$ and a unit vector $g_0\in H^{\infty}$ such that
\begin{enumerate}
\item $\|N g_0\|=\|N\|=1$,
\item $\|f_0 - g_0\|<2\sqrt{\epsilon}$, and
\item $\|T - N\|<7\sqrt{\epsilon}$.
\end{enumerate}
\end{corollary}

A similar statement holds if we replace $\mathscr{W}(H^{\infty})$ by the ideal of $p$-summing operators of $\clb(H^\infty)$. We refer the reader to  \cite[Chapter 2]{Diestel} for $p$-summing operators on Banach spaces.

Now we comment on Asplund spaces. Recall that a Banach space $X$ is called an \textit{Asplund space} if, whenever $g$ is a convex continuous function on an open convex subset $V$ of $X$, the set of all points of $V$ where $g$ is Fr\'{e}chet differentiable form a dense $G_{\delta}$-subset of $V$. This notion was introduced by Asplund under the name \textit{strong differentiability space} \cite{Asplund}. The following fundamental theorem gives a satisfactory classification of Asplund spaces \cite{Namioka, Stegall}:

\begin{theorem}
Let $X$ be a Banach space. Then the following conditions are equivalent:
\begin{enumerate}\label{Thm: Asp}
\item $X$ is an Asplund space;
\item $X^*$ has the Radon-Nikod\v{y}m property;
\item every separable subspace of $X$ has a separable dual.			
\end{enumerate}	
\end{theorem}
We refer to \cite{Diestel 2} for the notion of Radon-Nikod\v{y}m property. We now remark that the disc algebra $A(\bd)$ (a subalgebra of $H^\infty$) consisting of bounded analytic functions that are continuous on $\overline{\bd}$ is separable. However, its dual $A(\bd)^*$ is nonseparable. Consequently, Theorem \ref{Thm: Asp} implies that $H^\infty$ is not Asplund. Therefore, Theorem \ref{thm:main}, that is, the Bishop-Phelps-Bollob\'{a}s property for $\clb(H^\infty)$ does not follow from the existing theory of Asplund spaces (cf. \cite{Cascales}). For more information and recent development on Asplund spaces, we refer the reader to \cite{Brooker1,CFMM}.

\vspace{0.3in}

\noindent\textit{Acknowledgement:} We would like to thank the referee for pointing out an error in the earlier version of the manuscript and for the valuable comments and suggestions. We are also thankful to Professor Gilles Pisier for his helpful remarks on Asplund spaces. The research of first author is supported by the Department of Mathematics, IIT Bombay, India. The research of second author is supported in part by the Theoretical Statistics and Mathematics Unit, Indian Statistical Institute, Bangalore Centre, India and the INSPIRE grant of Dr. Srijan Sarkar (Ref: DST/INSPIRE/04/2019/000769), Department of Science \& Technology (DST), Government of India. The third author is supported in part by the Core Research Grant (CRG/2019/000908), by SERB, Department of Science \& Technology (DST), and NBHM (NBHM/R.P.64/2014), Government of India. The research of the fourth author is supported by the NBHM postdoctoral fellowship, Department of Atomic Energy (DAE), Government of India (File No: 0204/3/2020/R$\&$D-II/2445).

\bibliographystyle{amsplain}

\end{document}